\documentclass{amsart}
\usepackage{amssymb,amsmath,amsthm}
\usepackage[dvips]{graphicx}

\usepackage{graphicx,leftidx,undertilde}      

\usepackage{amsmath}

\usepackage{amssymb}
\let\oldmarginpar\marginpar
\renewcommand\marginpar[1]{\-\oldmarginpar[\raggedleft\footnotesize #1]%
{\raggedright\footnotesize #1}}

\begin{document}

\newtheorem{theorem}{Theorem}[section]
\newtheorem{corollary}[theorem]{Corollary}
\newtheorem{lemma}[theorem]{Lemma}
\newtheorem{proposition}[theorem]{Proposition}
\theoremstyle{definition}
\newtheorem{definition}[theorem]{Definition}
\theoremstyle{remark}
\newtheorem{remark}[theorem]{Remark}
\theoremstyle{definition}
\newtheorem{example}[theorem]{Example}
\theoremstyle{acknowledgment}
\newtheorem{acknowledgment}[theorem]{Acknowledgment}

\numberwithin{equation}{section}

\title[Extremities for statistical submanifolds]{Extremities for statistical submanifolds in \\ Kenmotsu statistical manifolds}

\author[A.N. Siddiqui]{Aliya Naaz Siddiqui}
\address{Department of Mathematics, Faculty of Natural Sciences, Jamia Millia Islamia, New Delhi - 110025, India}
\email{aliya148040@st.jmi.ac.in}

\author[Y.J. Suh]{Young Jin Suh}
\address{Department of Mathematics and RIRCM, College of Natural Sciences, Kyungpook National University, Daegu 41566, South Korea}
\email{yjsuh@knu.ac.kr}

\author[O. Bahad${\i}$r]{O$\breve{g}$uzhan Bahad${\i}$r}
\address{Department of Mathematics, Faculty of Science and Letters, Kahramanmaras Sutcu Imam University, Kahrmanmaras, TURKEY}
\email{oguzbaha@gmail.com}



\begin{abstract}
Kenmotsu geometry is a valuable part of contact geometry with nice applications in other fields such as theoretical physics. In this article, we study the statistical counterpart of a Kenmotsu manifold, that is, Kenmotsu statistical manifold with some related examples. We investigate some statistical curvature properties of Kenmotsu statistical manifolds. We prove that a Kenmotsu statistical manifold is not a Ricci-flat statistical manifold with an example. Finally, we prove a very well-known Chen-Ricci inequality for statistical submanifolds in Kenmotsu statistical manifolds of constant $\phi-$sectional curvature by adopting optimization techniques on submanifolds. This article ends with some concluding remarks.
\end{abstract}

\maketitle

\section{Introduction}

The geometry of a manifold has an important place in statistics as the statistical model often forms a geometrical manifold. The notion of statistical manifolds was first initiated by Amari in \cite{l} and applied by Lauritzen in \cite{LS}.

A Riemannian manifold $(\overline{B}, \overline{g})$ with a Riemannian metric $\overline{g}$ is said to be a statistical manifold $(\overline{B}, \overline{\nabla}, \overline{g})$ if $\overline{\nabla}\overline{g}$ is symmetric and a pair of torsion-free affine connections $\overline{\nabla}$ and $\overline{\nabla}^{\ast}$ on $\overline{B}$ satisfies
\begin{eqnarray*}
\mathrm{G}_{1} \overline{g}(\mathrm{E}_{1}, \mathrm{F}_{1}) = \overline{g}(\overline{\nabla}_{\mathrm{G}_{1}} \mathrm{E}_{1}, \mathrm{F}_{1}) + \overline{g}(\mathrm{E}_{1}, \overline{\nabla}_{\mathrm{G}_{1}}^{\ast}\mathrm{F}_{1}),
\end{eqnarray*}
for any $\mathrm{E}_{1}, \mathrm{F}_{1}, \mathrm{G}_{1} \in \Gamma(T\overline{B})$. Here $\overline{\nabla}^{\ast}$ is called the dual connection on $\overline{B}$.

\begin{remark}\label{g}
We have
\begin{enumerate}
\item[(1)] $\overline{\nabla} = (\overline{\nabla}^{\ast})^{\ast}$.
\item[(2)] $2\overline{\nabla}^{\overline{g}} = \overline{\nabla} + \overline{\nabla}^{\ast}$,\\
where $\overline{\nabla}^{\overline{g}}$ is the Levi-Civita connection of $\overline{g}$ on $\overline{B}$.
\item[(3)] if $(\overline{\nabla}, \overline{g})$ is a statistical structure on $\overline{B}$, then $(\overline{\nabla}^{\ast}, \overline{g})$ is again a statistical structure.
\end{enumerate}
\end{remark}

For a statistical manifold $(\overline{B}, \overline{\nabla}, \overline{g})$, set $\overline{\mathbb{K}} = \overline{\nabla} - \overline{\nabla}^{\overline{g}}$. Then $\overline{\mathbb{K}} \in \Gamma(T\overline{B}^{(1,2)})$
\begin{eqnarray*}
\overline{\mathbb{K}}_{\mathrm{E}_{1}}\mathrm{F}_{1} = \overline{\mathbb{K}}_{\mathrm{F}_{1}}\mathrm{E}_{1},
\end{eqnarray*}
and
\begin{eqnarray*}
\overline{g}(\overline{\mathbb{K}}_{\mathrm{E}_{1}}\mathrm{F}_{1}, \mathrm{G}_{1}) = \overline{g}(\mathrm{F}_{1}, \overline{\mathbb{K}}_{\mathrm{E}_{1}}\mathrm{G}_{1})
\end{eqnarray*}
hold for any $\mathrm{E}_{1}, \mathrm{F}_{1}, \mathrm{G}_{1} \in \Gamma(T\overline{B})$. Conversely, if $\overline{\mathbb{K}}$ satisfies above relations, then $(\overline{B}, \overline{\nabla} = \overline{\nabla}^{\overline{g}} + \overline{\mathbb{K}}, \overline{g})$ becomes a statistical manifold and write $\overline{\mathbb{K}}_{\mathrm{E}_{1}}\mathrm{F}_{1}$ as $\overline{\mathbb{K}}(\mathrm{E}_{1},\mathrm{F}_{1})$.

The statistical curvature tensor field $\overline{S}^{\overline{\nabla}, \overline{\nabla}^{\ast}} = \overline{S}$ with respect to $\overline{\nabla}$ and $\overline{\nabla}^{\ast}$ of $(\overline{B}, \overline{\nabla}, \overline{g})$ is \cite{f2}
\begin{eqnarray}\label{2.1}
\overline{S}(\mathrm{E}_{1}, \mathrm{F}_{1})\mathrm{G}_{1} = \overline{R}^{\overline{g}}(\mathrm{E}_{1}, \mathrm{F}_{1})\mathrm{G}_{1} + [\overline{\mathbb{K}}_{\mathrm{E}_{1}}, \overline{\mathbb{K}}_{\mathrm{F}_{1}}]\mathrm{G}_{1},
\end{eqnarray}
for any $\mathrm{E}_{1}, \mathrm{F}_{1}, \mathrm{G}_{1} \in \Gamma(T\overline{B})$. Here $\overline{R}^{\overline{g}}$ denotes the curvature tensor field with respect to $\overline{\nabla}^{\overline{g}}$.

K. Kenmotsu \cite{ken} studied the third class (that is, the warped product spaces $B \times_{s} M$, where $B$ is a line and $M$ a Kaehlerian manifold) in Tanno's classification of connected almost contact metric manifolds whose automorphism group has a maximum dimension. He analysed the properties of $B \times_{s} M$ and characterized it by tensor equations. Nowadays such a manifold is known by Kenmotsu manifold. Lately, Furuhata et al. \cite{f2} studied the statistical counterpart of a Kenmotsu manifold and introduced the notion of Kenmotsu statistical manifolds by putting an affine connection on a Kenmotsu manifold. They gave a method how to form a Kenmotsu statistical manifold as the warped product of a holomorphic statistical manifold \cite{f} and a line.

In 1993, B.-Y. Chen \cite{chen2} initiated the study to establish some bonds between the intrinsic and extrinsic invariants of submanifolds. He \cite{chen} proved a relation between the main extrinsic invariant squared mean curvature and the main intrinsic invariant Ricci curvature for any submanifold in real space forms. This inequality is celebrated as Chen-Ricci inequality. Since then numerous geometers obtained the similar inequalities for different classes of submanifolds and ambient spaces. T. Oprea \cite{toprea} used optimization techniques applied in the setup of Riemannian geometry to derive the Chen-Ricci inequality. With a pair of dual connections, Aydin et al. \cite{a} derived a Chen-Ricci inequality for statistical submanifolds in a statistical manifold of constant curvature. Also, A. Mihai et al. \cite{a2} established a similar inequality with respect to a sectional curvature of the ambient Hessian manifold. Recently, Siddiqui et al. studied statistical warped products as submanifolds of statistical manifolds. For statistical warped products statistically immersed in a statistical manifold of constant curvature, they proved Chen's inequality involving scalar curvature, the squared mean curvature, and the Laplacian of warping function (with respect to the Levi--Civita connection) in \cite{al}. Recently, Siddiqui et al. \cite{aaliya} obtained inequality for Ricci curvature of submanifolds in statistical manifolds of constant(quasi-constant) curvature by Oprea's optimization method.

Motivated by the above studies, we establish the Chen-Ricci inequality with respect to a statistical sectional curvature of the ambient Kenmotsu statistical manifold of constant $\phi-$sectional curvature. Furthermore, we study the equality cases by using the most interesting result, namely, optimization on Riemannian submanifolds as follows:

Let $(N, g)$ be a Riemannian submanifold of a Riemannian manifold $(\overline{B}, \overline{g})$ and $f : \overline{B} \rightarrow \mathbb{R}$ be a differentiable function. Let:
 \begin{eqnarray}\label{111.}
 \min\limits_{x_{0} \in N} f(x_{0})
  \end{eqnarray}
 be the constrained extremum problem.

\begin{theorem}\label{oth}
 \cite{TO} If $x \in N$ is the solution of the problem (\ref{111.}), then
\begin{enumerate}
\item[(1)] $(grad \hspace{0.1 cm} f)(x) \in T_{x}^{\perp}N$,
\item[(2)] the bilinear form $\Theta : T_{x}N \times T_{x}N \rightarrow \mathbb{R}$,
$$\Theta(E, F) = Hess_{f}(E, F) + \overline{g}(h^{'}(E, F), (grad \hspace{0.1 cm} f)(x))$$
\end{enumerate}
is positive semi-definite, where $h^{'}$ is the second fundamental form of $N$ in $\overline{B}$ and $grad \hspace{0.1 cm} f$ denotes the gradient of $f$.
\end{theorem}

We prove the following main theorem of the article:

\begin{theorem}\label{th}
Let $(\overline{B}(\overline{c}), \overline{\nabla}, \overline{g}, \phi, \xi)$ be a $(2s+1)-$dimensional Kenmotsu statistical manifold of constant $\phi-$sectional curvature $\overline{c}$ and $(N, \nabla, g)$ be an $(k+1)-$dimensional statistical submanifold in $\overline{B}(\overline{c})$. Then
\begin{enumerate}
\item[(1)] For each unit vector $\mathrm{E}_{1} \in T_{p}N$, $p \in N$,
\begin{eqnarray}\label{r}
Ric^{\nabla, \nabla^{\ast }}(\mathrm{E}_{1}) &\geq& 2Ric^{0}(\mathrm{E}_{1}) - \{\frac{3(\overline{c}+1)}{4}||\mathcal{P} \mathrm{E}_{1}||^{2} + \frac{k}{4}[(\overline{c}+1)(1-g^{2}(\mathrm{E}_{1},\xi))-4]\} \nonumber \\
&&- \frac{(k+1)^{2}}{8} [||\mathcal{H}||^{2} + ||\mathcal{H}^{\ast }||^{2}],
\end{eqnarray}
where $Ric^{0}$ denotes the Ricci curvature with respect to Levi-Civita connection.
\item[(2)] Moreover, the equality holds in the inequality (\ref{r}) if and only if $$2h(\mathrm{E}_{1}, \mathrm{E}_{1}) = (k+1) \mathcal{H}(p),$$ $$2h^{\ast}(\mathrm{E}_{1}, \mathrm{E}_{1}) = (k+1) \mathcal{H}^{\ast}(p)$$ and $$h(\mathrm{E}_{1}, \mathrm{F}_{1}) =0,$$ $$h^{\ast }(\mathrm{E}_{1}, \mathrm{F}_{1})=0,$$ for all $\mathrm{F}_{1} \in T_{p}N$ orthogonal to $\mathrm{E}_{1}$.
\end{enumerate}
\end{theorem}

\section{Preliminaries}

Let $(N, \nabla, g)$ and $(\overline{B}, \overline{\nabla}, \overline{g})$ be two statistical manifolds. An immersion $\iota : N \rightarrow \overline{B}$ is called a statistical immersion if $(\nabla, g)$ coincides with the induced statistical structure, that is, if
\begin{eqnarray}
(\nabla_{\mathrm{E}_{1}} g)(\mathrm{F}_{1}, \mathrm{G}_{1}) = (\nabla_{\mathrm{F}_{1}} g)(\mathrm{E}_{1}, \mathrm{G}_{1})
\end{eqnarray}
holds for any $\mathrm{E}_{1}, \mathrm{F}_{1}, \mathrm{G}_{1} \in \Gamma(T\overline{B})$. If a statistical immersion exists between two statistical manifolds, then we call $(N, \nabla, g)$ as a statistical submanifold in $(\overline{B}, \overline{\nabla}, \overline{g})$. Then the Gauss formulae are \cite{vos}
\begin{eqnarray}
\overline{\nabla}_{\mathrm{E}_{1}}\mathrm{F}_{1} &=& \nabla_{\mathrm{E}_{1}}\mathrm{F}_{1} + h(\mathrm{E}_{1}, \mathrm{F}_{1}),
\end{eqnarray}
and
\begin{eqnarray}
\overline{\nabla}_{\mathrm{E}_{1}}^{\ast}\mathrm{F}_{1} &=& \nabla_{\mathrm{E}_{1}}^{\ast}\mathrm{F}_{1} + h^{\ast}(\mathrm{E}_{1}, \mathrm{F}_{1}),
\end{eqnarray}
for any $\mathrm{E}_{1}, \mathrm{F}_{1} \in \Gamma(TN)$. We denote the dual connections on $\Gamma(TN^{\perp})$ by $D^{\perp}$ and $D^{\perp \ast}$. Then the corresponding Weingarten formulae are \cite{vos}:
\begin{eqnarray}
\overline{\nabla}_{\mathrm{E}_{1}}\mathcal{U} &=& - A_{\mathcal{U}}\mathrm{E}_{1} + D_{\mathrm{E}_{1}}^{\perp}\mathcal{U},
\end{eqnarray}
and
\begin{eqnarray}
\overline{\nabla}_{\mathrm{E}_{1}}^{\ast}\mathcal{U} &=& - A_{\mathcal{U}}^{\ast}\mathrm{E}_{1} + D_{\mathrm{E}_{1}}^{\perp \ast }\mathcal{U},
\end{eqnarray}
for any $\mathrm{E}_{1} \in \Gamma(TN)$ and $\mathcal{U} \in \Gamma(TN^{\perp})$. The symmetric and bilinear embedding curvature tensors of $N$ in $\overline{B}$ with respect to $\overline{\nabla}$ and $\overline{\nabla}^{\ast}$ are denoted by $h$ and $h^{\ast}$ respectively. The linear transformations $A_{\mathcal{U}}$ and $A_{\mathcal{U}}^{\ast}$ are given by \cite{vos}
\begin{eqnarray}
\overline{g}(h(\mathrm{E}_{1}, \mathrm{F}_{1}), \mathcal{U}) &=& g(A_{\mathcal{U}}^{\ast}\mathrm{E}_{1}, \mathrm{F}_{1}),
\end{eqnarray}
and
\begin{eqnarray}
\overline{g}(h^{\ast}(\mathrm{E}_{1}, \mathrm{F}_{1}), \mathcal{U}) &=& g(A_{\mathcal{U}} \mathrm{E}_{1}, \mathrm{F}_{1}),
\end{eqnarray}
for any $\mathrm{E}_{1}, \mathrm{F}_{1} \in \Gamma(TN)$ and $\mathcal{U} \in \Gamma(TN^{\perp})$.

We denote the Riemannian curvature tensor fields with respect to $\overline{\nabla}$ and $\overline{\nabla}^{\ast}$ by $\overline{R}$ and $\overline{R}^{\ast}$, respectively. Also, $R$ and $R^{\ast}$ are the the Riemannian curvature tensor fields with respect to the induced connections $\nabla$ and $\nabla^{\ast}$ of $\overline{\nabla}$ and $\overline{\nabla}^{\ast}$, respectively. Then the Gauss equations are \cite{vos}
\begin{eqnarray}\label{i111}
\overline{g}(\overline{R}(\mathrm{E}_{1}, \mathrm{F}_{1}) \mathrm{G}_{1}, \mathrm{H}_{1}) &=& g(R(\mathrm{E}_{1}, \mathrm{F}_{1}) \mathrm{G}_{1}, \mathrm{H}_{1}) + \overline{g}(h(\mathrm{E}_{1}, \mathrm{G}_{1}), h^{\ast}(\mathrm{F}_{1}, \mathrm{H}_{1})) \nonumber \\
&&- \overline{g}(h^{\ast}(\mathrm{E}_{1}, \mathrm{H}_{1}), h(\mathrm{F}_{1}, \mathrm{G}_{1})),
\end{eqnarray}
and
\begin{eqnarray}\label{i1}
\overline{g}(\overline{R}^{\ast}(\mathrm{E}_{1}, \mathrm{F}_{1}) \mathrm{G}_{1}, \mathrm{H}_{1}) &=& g(R^{\ast}(\mathrm{E}_{1}, \mathrm{F}_{1}) \mathrm{G}_{1}, \mathrm{H}_{1}) + \overline{g}(h^{\ast}(\mathrm{E}_{1}, \mathrm{G}_{1}), h(\mathrm{F}_{1}, \mathrm{H}_{1})) \nonumber \\
&&- \overline{g}(h(\mathrm{E}_{1}, \mathrm{H}_{1}), h^{\ast}(\mathrm{F}_{1}, \mathrm{G}_{1})),
\end{eqnarray}
for any $\mathrm{E}_{1}, \mathrm{F}_{1}, \mathrm{G}_{1}, \mathrm{H}_{1} \in \Gamma(TN)$. Also, we have
\begin{eqnarray}\label{r1}
2\overline{S} = \overline{R} + \overline{R}^{\ast},
\end{eqnarray}
and
\begin{eqnarray}\label{r2}
2S = R + R^{\ast},
\end{eqnarray}
where $S^{\nabla, \nabla^{\ast}} = S \in \Gamma(TN^{(1,3)})$ denotes the statistical curvature tensor field with respect to $\nabla$ and $\nabla^{\ast}$ of $(N, \nabla, g)$.

In general, one cannot define a sectional curvature with respect to the dual connections (which are not metric) by the standard definitions. However, B. Opozda \cite{article.82,article.822} defined a sectional curvature on a statistical manifold as follows:
\begin{eqnarray}\label{r3}
\overline{\mathcal{K}}(\mathrm{E}_{1} \wedge \mathrm{F}_{1}) &=& g(\overline{S}(\mathrm{E}_{1}, \mathrm{F}_{1})\mathrm{F}_{1}, \mathrm{E}_{1})\nonumber  \\
&=& \frac{1}{2}(g(\overline{R}(\mathrm{E}_{1}, \mathrm{F}_{1}) \mathrm{F}_{1}, \mathrm{E}_{1}) + g(\overline{R}^{\ast}(\mathrm{E}_{1}, \mathrm{F}_{1})\mathrm{F}_{1}, \mathrm{E}_{1})),
\end{eqnarray}
for any orthonormal vectors $\mathrm{E}_{1}, \mathrm{F}_{1} \in \Gamma(T\overline{B})$.

Suppose that dim$(N) = k$ and dim$(\overline{B}) = s$. Given a local orthonormal frame $\{e_{1}, \dots, e_{k}\}$ of $T_{p}N$ and $\{e_{k+1}, \dots, e_{s}\}$ of $T_{p}N^{\perp}$, $p \in N$. Then the mean curvature vectors $\mathcal{H}$ and $\mathcal{H}^{\ast }$ are
\begin{eqnarray*}
k\mathcal{H} &=& \sum_{i = 1}^{k} h(e_{i}, e_{i}),\\
k\mathcal{H}^{\ast } &=& \sum_{i = 1}^{k}h^{\ast }(e_{i}, e_{i}).
\end{eqnarray*}
From Remark \ref{g}, it can be easily verified that
\begin{eqnarray*}
2h^{0} = h + h^{\ast},
\end{eqnarray*}
and hence
\begin{eqnarray*}
2\mathcal{H}^{0} = \mathcal{H} + \mathcal{H}^{\ast},
\end{eqnarray*}
where $h^{0}$ and $\mathcal{H}^{0}$ are the second fundamental form and the mean curvature field with respect to $\overline{\nabla}^{\overline{g}}$ respectively. We set
\begin{eqnarray*}
h_{ij}^{r} &=& \overline{g}(h(e_{i}, e_{j}), e_{r}),\\
h_{ij}^{\ast r} &=& \overline{g}(h^{\ast }(e_{i}, e_{j}), e_{r}),
\end{eqnarray*}
for $i, j = \{1, \dots, m\}$ and $r = \{k+1, \dots, n\}$. Then the squared norm of mean curvature vectors are
\begin{eqnarray*}
k^{2}||\mathcal{H}||^{2} &=& \sum_{r = k+1}^{s}\bigg(\sum_{i = 1}^{k}h_{ii}^{r}\bigg)^{2}, \\
k^{2}||\mathcal{H}^{\ast }||^{2} &=& \sum_{r = k+1}^{s} \bigg(\sum_{i = 1}^{k}h_{ii}^{\ast  r}\bigg)^{2}.
\end{eqnarray*}

Kenmotsu geometry is a branch of differential geometry with nice applications in mechanics of dynamical systems with time dependent Hamiltonian, geometrical optics, thermodynamics and geometric quantization. Also, the study of submanifolds in Kenmotsu ambient spaces is a valuable subject in Kenmotsu geometry, which has been analysed by many geometers.

\begin{definition}
\cite{f2} Let $(\overline{B}, \overline{\nabla}, \overline{g}, \phi, \xi)$ be a Kenmotsu manifold. A quadruplet $(\overline{B}, \overline{\nabla} = \overline{\nabla}^{\overline{g}} + \overline{\mathbb{K}}, \overline{g}, \phi, \xi)$ is called a Kenmotsu statistical manifold if $(\overline{\nabla}, \overline{g})$ is a statistical structure on $\overline{B}$ and the formula
\begin{eqnarray}
\overline{\mathbb{K}}(\mathrm{E}_{1}, \phi \mathrm{F}_{1}) = - \phi \overline{\mathbb{K}}(\mathrm{E}_{1}, \mathrm{F}_{1})
\end{eqnarray}
holds for any $\mathrm{E}_{1}, \mathrm{F}_{1} \in \Gamma(T\overline{B})$. Here we call $(\overline{\nabla}, \overline{g}, \phi, \xi)$ a Kenmotsu statistical structure on $\overline{B}$.
\end{definition}

A Kenmotsu statistical manifold $(\overline{B}, \overline{\nabla}, \overline{g}, \phi, \xi)$ is said to be of constant $\phi-$sectional curvature $\overline{c} \in \mathbb{R}$ if \cite{f2}
\begin{eqnarray}\label{d}
\overline{R}(\mathrm{E}_{1}, \mathrm{F}_{1})\mathrm{G}_{1} &=& \frac{\overline{c}-3}{4} \{\overline{g}(\mathrm{F}_{1}, \mathrm{G}_{1})\mathrm{E}_{1} - \overline{g}(\mathrm{E}_{1}, \mathrm{G}_{1})\mathrm{F}_{1}\} \nonumber \\
&&+ \frac{\overline{c}+1}{4} \{\overline{g}(\phi \mathrm{F}_{1}, \mathrm{G}_{1})\phi \mathrm{E}_{1} - \overline{g}(\phi \mathrm{E}_{1}, \mathrm{G}_{1})\phi \mathrm{F}_{1}  \nonumber \\
&&- 2\overline{g}(\phi \mathrm{E}_{1}, \mathrm{F}_{1})\phi \mathrm{G}_{1} - \overline{g}(\mathrm{F}_{1}, \xi)\overline{g}(\mathrm{G}_{1}, \xi)\mathrm{E}_{1}  \nonumber \\
&&+ \overline{g}(\mathrm{E}_{1}, \xi)\overline{g}(\mathrm{G}_{1}, \xi)\mathrm{F}_{1} + \overline{g}(\mathrm{F}_{1}, \xi)\overline{g}(\mathrm{G}_{1}, \mathrm{E}_{1})\xi \nonumber \\
&&- \overline{g}(\mathrm{E}_{1}, \xi)\overline{g}(\mathrm{G}_{1}, \mathrm{F}_{1})\xi \},
\end{eqnarray}
for $\mathrm{E}_{1}, \mathrm{F}_{1}, \mathrm{G}_{1} \in \Gamma(T\overline{B})$. We denote it by $\overline{B}(\overline{c})$.

Let $(N, \nabla, g)$ be a statistical submanifold in a Kenmotsu statistical manifold $(\overline{B}, \overline{\nabla}, \overline{g}, \phi, \xi)$. Then, the corresponding Gauss equation is given by (for details see \cite{f})
\begin{eqnarray}\label{i}
R(\mathrm{E}_{1}, \mathrm{F}_{1})\mathrm{G}_{1} &=& \frac{\overline{c}-3}{4} \{g(\mathrm{F}_{1}, \mathrm{G}_{1})\mathrm{E}_{1} - g(\mathrm{E}_{1}, \mathrm{G}_{1})\mathrm{F}_{1}\} \nonumber \\
&&+ \frac{\overline{c}+1}{4} \{g(\phi \mathrm{F}_{1}, \mathrm{G}_{1})\phi \mathrm{E}_{1} - g(\phi \mathrm{E}_{1}, \mathrm{G}_{1})\phi \mathrm{F}_{1}  \nonumber \\
&&- 2g(\phi \mathrm{E}_{1}, \mathrm{F}_{1})\phi \mathrm{G}_{1} - g(\mathrm{F}_{1}, \xi)g(\mathrm{G}_{1}, \xi)\mathrm{E}_{1}  \nonumber \\
&&+ g(\mathrm{E}_{1}, \xi)g(\mathrm{G}_{1}, \xi)\mathrm{F}_{1} + g(\mathrm{F}_{1}, \xi)g(\mathrm{G}_{1}, \mathrm{E}_{1})\xi \nonumber \\
&&- g(\mathrm{E}_{1}, \xi)g(\mathrm{G}_{1}, \mathrm{F}_{1})\xi \} \nonumber \\
&&+ \frac{1}{2}\{ \mathrm{A}_{h(\mathrm{F}_{1}, \mathrm{G}_{1})}\mathrm{E}_{1} + \mathrm{A}_{h^{\ast}(\mathrm{F}_{1}, \mathrm{G}_{1})}^{\ast}\mathrm{E}_{1} \}\nonumber \\
&&- \frac{1}{2}\{ \mathrm{A}_{h(\mathrm{E}_{1}, \mathrm{G}_{1})}\mathrm{F}_{1} + \mathrm{A}_{h^{\ast}(\mathrm{E}_{1}, \mathrm{G}_{1})}^{\ast}\mathrm{F}_{1}\},
\end{eqnarray}
for any $\mathrm{E}_{1}, \mathrm{F}_{1}, \mathrm{G}_{1} \in \Gamma(TN)$.

Any $\mathrm{E}_{1} \in \Gamma(TN)$ can be decomposed uniquely into its tangent and normal parts $P\mathrm{E}_{1}$ and $C\mathrm{E}_{1}$ respectively,
\begin{eqnarray*}
\phi \mathrm{E}_{1} = P\mathrm{E}_{1} + C\mathrm{E}_{1}.
\end{eqnarray*}
The squared norm of $P$ is defined by
\begin{eqnarray*}
||P||^{2} = \sum_{i,j=1}^{k+1} g^{2}(Pe_{i}, e_{j}),
\end{eqnarray*}
where $\{e_{1}, \dots, e_{k+1}\}$ denotes a local orthonormal frame of $T_{p}N$.

A statistical submanifold $(N, \nabla, g)$ in a Kenmotsu statistical manifold $(\overline{B}, \overline{\nabla}, \overline{g}, \phi, \xi)$ is said to be invariant, $C=0$, (respectively, anti-invariant, $P=0$) if $\phi \mathrm{E}_{1} \in \Gamma(TN)$ for any $\mathrm{E}_{1} \in \Gamma(TN)$ (respectively, $\phi \mathrm{E}_{1} \in \Gamma(TN^{\perp})$ for any $\mathrm{E}_{1} \in \Gamma(TN)$).

\section{Statistical Curvature Properties}

Before going to prove the results, first we recall the following results of \cite{f2}:

\begin{proposition}\label{p}
\cite{f2} Let $(\tilde{B}, \tilde{g}, J)$ be an almost Hermitian manifold.
Set $\overline{B} = \tilde{B} \times \mathbb{R}$, $\overline{g} = e^{2\alpha}g + (d\alpha)^{2}$,
$\xi = \frac{\partial}{\partial \alpha} \in \Gamma(T\overline{B})$ and define
$\phi \in \Gamma(T\overline{B}^{(1,1)})$ by $\phi \mathrm{E}_{2} = J\mathrm{E}_{2}$ for any $\mathrm{E}_{2} \in \Gamma(T\tilde{B})$ and $\phi \xi = 0$. Then,
\begin{enumerate}
\item[(1)] The triple $(\overline{g}, \phi, \xi)$ is an almost contact metric structure on $\overline{B}$.
\item[(2)] The pair $(\tilde{g}, J)$ is a \textrm{K$\ddot{a}$hler} structure on $\tilde{B}$ if and only if the triple $(\overline{g}, \phi, \xi)$ is a Kenmotsu structure on $\overline{B}$.
\end{enumerate}
\end{proposition}

\begin{theorem}\label{p1}
\cite{f2} Let $(\tilde{B}, \tilde{g}, J)$ be a \textrm{K$\ddot{a}$hler} manifold, $(\overline{B} = \tilde{B} \times \mathbb{R}, \overline{g}, \phi, \xi)$ the Kenmotsu manifold as in  Proposition \ref{p}, and $(\overline{\nabla} = \nabla^{\overline{g}} + \overline{\mathbb{K}}, \overline{g})$
a statistical structure on $\overline{B}$. Define $\mathcal{A} \in \Gamma(T\overline{B}^{(0,2)} \otimes T\tilde{B})$, $\Theta \in \Gamma(T\overline{B}^{(0,2)})$ and $\mathbb{K} \in \Gamma(T\tilde{B}^{(1,2)})$ by
\begin{eqnarray*}
\overline{\mathbb{K}}(\mathrm{E}_{1}, \mathrm{F}_{1}) = \mathcal{A}(\mathrm{E}_{1}, \mathrm{F}_{1}) + \Theta(\mathrm{E}_{1}, \mathrm{F}_{1})\xi, \hspace{0.5 cm} \textrm{and} \hspace{0.5 cm} \mathbb{K}(\mathrm{E}_{2}, \mathrm{F}_{2}) = \mathcal{A}(\mathrm{E}_{2}, \mathrm{F}_{2}),
\end{eqnarray*}
for $\mathrm{E}_{1}, \mathrm{F}_{1} \in \Gamma(T\overline{B})$ and $\mathrm{E}_{2}, \mathrm{F}_{2} \in \Gamma(T\tilde{B})$. Then the following conditions are equivalent:
\begin{enumerate}
\item[(1)] $(\overline{\nabla}, \overline{g}, \phi, \xi)$ is a Kenmotsu statistical structure on $\overline{B}$.
\item[(2)] $(\tilde{\nabla} = \nabla^{\tilde{g}} + \mathbb{K}, \tilde{g}, J)$ is a holomorphic statistical structure on $N$,
and the formulae $\mathcal{A}(\mathrm{E}_{1}, \xi) = 0$, $\Theta(\mathrm{E}_{1}, \mathrm{F}_{2}) = 0$ hold for $\mathrm{E}_{1} \in \Gamma(T\overline{B})$ and $\mathrm{F}_{2} \in \Gamma(T\tilde{B})$.
\end{enumerate}
\end{theorem}

\begin{proposition}\label{p2}
\cite{f2} Let $(\tilde{B}, \tilde{\nabla} = \nabla^{\tilde{g}} + \mathbb{K}, \tilde{g}, J)$ be a holomorphic statistical manifold, and
$(\overline{B} = \tilde{B} \times \mathbb{R}, \overline{g}, \phi, \xi)$ the Kenmotsu manifold as in  Proposition \ref{p}.
For any $\beta \in C^{\infty}(\overline{B})$, define $\overline{\mathbb{K}} \in \Gamma(T\overline{B}^{(1,2)})$ by
\begin{eqnarray*}
\overline{\mathbb{K}}(\mathrm{E}_{2}, \mathrm{F}_{2}) = \mathbb{K}(\mathrm{E}_{2}, \mathrm{F}_{2}), \hspace{0.5 cm} \overline{\mathbb{K}}(\mathrm{E}_{2}, \xi) = \overline{\mathbb{K}}(\xi, \mathrm{E}_{2}) = 0, \hspace{0.5 cm} \textrm{and} \hspace{0.5 cm} \overline{\mathbb{K}}(\xi, \xi) = \beta \xi,
\end{eqnarray*}
for any $\mathrm{E}_{2}, \mathrm{F}_{2} \in \Gamma(T\tilde{B})$. Then $(\overline{\nabla} = \nabla^{\overline{g}} + \overline{\mathbb{K}}, \overline{g}, \phi, \xi)$ is a Kenmotsu statistical structure on $\overline{B}$.
\end{proposition}

Here we construct an easy example on Kenmotsu statistical manifold by using above Propositions \ref{p} and \ref{p2}. This is as follows:

\begin{example}\label{ex}
Let us consider a holomorphic statistical manifold $(\tilde{B}^{2}, \tilde{\nabla} = \nabla^{\tilde{g}} + \mathbb{K}, \tilde{g}, J)$ \cite{12aliya}, where
\begin{eqnarray*}
\tilde{B}^{2} = \{^{t}(x, y) \in \mathbb{R}^{2}| x > 0\}, \hspace{0.5 cm} \tilde{g} = x[(dx)^{2} + (dy)^{2}], \\
J \partial_{1} = \partial_{2}, \hspace{0.5 cm} J \partial_{2} = - \partial_{1}, \hspace{0.5 cm} \partial_{i} = \frac{\partial}{\partial x^{i}} \hspace{0.2 cm}  for \hspace{0.2 cm} i = 1, 2, \\
\mathbb{K}(\partial_{1}, \partial_{1}) = -\lambda \partial_{1},  \hspace{0.5 cm} \mathbb{K}(\partial_{1}, \partial_{2}) = \mathbb{K}(\partial_{2}, \partial_{1}) = \lambda \partial_{2}, \hspace{0.5 cm} \mathbb{K}(\partial_{2}, \partial_{2}) = \lambda \partial_{1},
\end{eqnarray*}

The affine connections $\tilde{\nabla}$ on $\tilde{B}^{2}$ are defined by
\begin{eqnarray*}
\tilde{\nabla}_{\partial_{1}}\partial_{1} &=& (\frac{1}{2}(x)^{-1} - \lambda )\partial_{1},\\
\tilde{\nabla}_{\partial_{1}}\partial_{2} &=& \tilde{\nabla}_{\partial_{2}}\partial_{1} = (\frac{1}{2}(x)^{-1} + \lambda)\partial_{2},\\
\tilde{\nabla}_{\partial_{2}}\partial_{2} &=& - (\frac{1}{2}(x)^{-1} - \lambda )\partial_{1}.
\end{eqnarray*}

We take a product, that is, $\overline{B}^{3} = \tilde{B} \times \mathbb{R}$, and $(\overline{B}^{3}, \overline{g}, \phi, \xi)$ is the Kenmotsu manifold as given in  Proposition \ref{p}. For this, we define
\begin{eqnarray*}
\phi(x, y, \alpha) = (-y, x, 0), \hspace{0.5 cm} \xi = \frac{\partial}{\partial \alpha} = \partial_{\alpha}, \\
\phi \partial_{1} = -\partial_{2}, \hspace{0.5 cm} \phi \partial_{2} = \partial_{1}, \hspace{0.5 cm} \phi \xi = 0,\\
\overline{g} = e^{2\alpha}[(dx)^{2} + (dy)^{2}] + (d\alpha)^{2},
\end{eqnarray*}
where $(x, y, \alpha)$ denotes the coordinates of $\overline{B}^{3}$. For $\beta = 1$, we define a $(1,2)-$tensor field $\overline{\mathbb{K}} \in \Gamma(T\overline{B}^{3})$ by
\begin{eqnarray*}
\overline{\mathbb{K}}(\partial_{1}, \partial_{1}) = -\lambda \partial_{1}, \hspace{0.5 cm} \overline{\mathbb{K}}(\partial_{1}, \partial_{2}) = \overline{\mathbb{K}}(\partial_{2}, \partial_{1}) = \lambda \partial_{2}, \hspace{0.5 cm} \overline{\mathbb{K}}(\partial_{\alpha}, \partial_{\alpha}) = \partial_{\alpha},\\
\overline{\mathbb{K}}(\partial_{2}, \partial_{2}) = \lambda \partial_{1}, \hspace{0.5 cm} \overline{\mathbb{K}}(\partial_{i}, \partial_{\alpha}) = \overline{\mathbb{K}}(\partial_{\alpha}, \partial_{i}) = 0 \hspace{0.2 cm}  for \hspace{0.2 cm} i = 1,2.
\end{eqnarray*}

Thus, by  Proposition \ref{p2}, we conclude that
$(\overline{\nabla} = \nabla^{\overline{g}} + \overline{\mathbb{K}}, \overline{g}, \phi, \xi)$ is a Kenmotsu statistical structure on $\overline{B}^{3}$.
\end{example}

We recall the definition of Jacobi operator \cite{m} and give the following statistical version of the definition of Jacobi operator:

\begin{definition}\label{df}
Let $(\overline{B}, \overline{\nabla}, \overline{g})$ be a statistical manifold. For any tangent vector field $\mathrm{E}_{1}$ at $p \in \overline{B}$, the Jacobi operator $\overline{R}_{\mathrm{E}_{1}}$ is defined by
\begin{eqnarray}\label{000}
(\overline{R}_{\mathrm{E}_{1}}\mathrm{F}_{1})(p) = (\overline{R}(\mathrm{F}_{1},\mathrm{E}_{1})\mathrm{E}_{1})(p),
\end{eqnarray}
for any $\mathrm{F}_{1} \in \Gamma(T\overline{B})$.
\end{definition}

\begin{remark}
In particular, we replace $\mathrm{E}_{1}$ by $\xi$ in the equation  (\ref{000}), then we call $\overline{R}_{\xi}$ as structure Jacobi operator.
\end{remark}

We give the following propositions:

\begin{proposition}\label{pp}
Let $(\tilde{B}, \tilde{\nabla} = \nabla^{\tilde{g}} + \mathbb{K}, \tilde{g}, J)$ and $(\overline{B}, \overline{\nabla} = \nabla^{\overline{g}} + \overline{\mathbb{K}}, \overline{g}, \phi, \xi)$ be a holomorphic statistical manifold, and the Kenmotsu statistical manifold as in  Theorem \ref{p1}, respectively. Then the structure Jacobi operator is parallel with respect to $\overline{\nabla}$.
\end{proposition}

\begin{proof}
From (3.18) of \cite{f2} and (\ref{000}), we have $\overline{R}_{\xi}(\mathrm{E}_{2}) = \overline{R}(\mathrm{E}_{2}, \xi)\xi = -\mathrm{E}_{2}$, for any $\mathrm{E}_{2} \in \Gamma(T\tilde{B})$. For any $\mathrm{F}_{2} \in \Gamma(T\tilde{B})$, we have
\begin{eqnarray*}
(\overline{\nabla}_{\mathrm{F}_{2}}\overline{R}_{\xi})(\mathrm{E}_{2}) &=&  \overline{\nabla}_{\mathrm{F}_{2}}\overline{R}_{\xi}(\mathrm{E}_{2}) - \overline{R}_{\xi}(\overline{\nabla}_{\mathrm{F}_{2}}\mathrm{E}_{2})\\
&=& - \overline{\nabla}_{\mathrm{F}_{2}}\mathrm{E}_{2} + \overline{\nabla}_{\mathrm{F}_{2}}\mathrm{E}_{2} = 0.
\end{eqnarray*}
Hence, we get our assertion.
\end{proof}

\begin{remark}
In the same  Proposition \ref{pp}, one can prove that the structure Jacobi operator is parallel with respect to $\overline{\nabla}^{\ast}$ also.
\end{remark}

\begin{definition}
A statistical manifold is said to be a Ricci-flat statistical manifold if its Ricci curvature vanishes.
\end{definition}

\begin{proposition}\label{pp1}
Let $(\tilde{B}, \tilde{\nabla} = \nabla^{\tilde{g}} + \mathbb{K}, \tilde{g}, J)$ and $(\overline{B}, \overline{\nabla} = \nabla^{\overline{g}} + \overline{\mathbb{K}}, \overline{g}, \phi, \xi)$ be a holomorphic statistical manifold, and the Kenmotsu statistical manifold as in  Theorem \ref{p1}, respectively. If $\overline{B}$ is of constant $\phi-$sectional curvature $\overline{c}$ and dim$(\tilde{B})=2s$, then Ricci tensor $\tilde{Ric}$ of $(\tilde{B}, \tilde{\nabla}, \tilde{g})$ is given by $$\tilde{Ric} = e^{2\alpha}(\frac{(\overline{c}+1)(s +1)}{2})\tilde{g}$$ for any $s \in \mathbb{R}$. Furthermore, $\tilde{B}$ is Ricci-flat statistical manifold if $\overline{c}=-1$.
\end{proposition}

\begin{proof}
Let $\{\tilde{e}_{1}, \dots, \tilde{e}_{2s}\}$ be a local orthonormal frame of $T_{p}\tilde{B}$, $p \in \tilde{B}$. From Proposition 3.9. of \cite{f2}, we get
\begin{eqnarray*}
\tilde{Ric}(\mathrm{E}_{2}) &=& \sum_{i = 1}^{2s} \tilde{g}(\tilde{R}(\tilde{e}_{i}, \mathrm{E}_{2})\mathrm{E}_{2}, \tilde{e}_{i}) \\
&=& e^{2\alpha}(\frac{\overline{c}+1}{4})\sum_{i = 1}^{2s}\{\tilde{g}(\mathrm{E}_{2}, \mathrm{E}_{2})\tilde{g}(\tilde{e}_{i}, \tilde{e}_{i}) - \tilde{g}(\tilde{e}_{i}, \mathrm{E}_{2})\tilde{g}(\mathrm{E}_{2}, \tilde{e}_{i}) \\
&&+ \tilde{g}(J\mathrm{E}_{2}, \mathrm{E}_{2})\tilde{g}(J\tilde{e}_{i}, \tilde{e}_{i}) - 3 \tilde{g}(J\tilde{e}_{i}, \mathrm{E}_{2})\tilde{g}(J\mathrm{E}_{2}, \tilde{e}_{i})\}\\
&=& e^{2\alpha}(\frac{\overline{c}+1}{4})\{(2s - 1) ||\mathrm{E}_{2}||^{2} + 3\tilde{g}(J\mathrm{E}_{2}, J\mathrm{E}_{2})\}\\
&=& e^{2\alpha}(\frac{\overline{c}+1}{4})\{(2s + 2) ||\mathrm{E}_{2}||^{2}\}\\
&=& e^{2\alpha}(\frac{(\overline{c}+1)(s + 1)}{2})||\mathrm{E}_{2}||^{2},
\end{eqnarray*}
where $\tilde{R}$ denotes the statistical curvature tensor field of $(\tilde{\nabla}, \tilde{g})$, $s \in \mathbb{R}$ and $\mathrm{E}_{2} \in \Gamma(T\tilde{B})$.
If we take $\overline{c}=-1$, $\tilde{Ric} = 0$ implies that Ricci curvature of $\tilde{B}$ vanishes, and hence $\tilde{B}$ is Ricci-flat statistical manifold. This is the required assertion.
\end{proof}

\begin{proposition}\label{1p}
Let $(\tilde{B}, \tilde{\nabla} = \nabla^{\tilde{g}} + \mathbb{K}, \tilde{g}, J)$ and $(\overline{B}, \overline{\nabla} = \nabla^{\overline{g}} + \overline{\mathbb{K}}, \overline{g}, \phi, \xi)$ be a holomorphic statistical manifold, and the Kenmotsu statistical manifold as in  Proposition \ref{p2}, respectively. If $\mathbb{K}(\mathrm{E}_{2}, \mathrm{F}_{2}) = 0$, and $\overline{\mathbb{K}}(\mathrm{E}_{1}, \mathrm{F}_{1}) = \beta \overline{g}(\mathrm{E}_{1}, \xi)\overline{g}(\mathrm{F}_{1}, \xi)\xi$, for any $\mathrm{E}_{2}, \mathrm{F}_{2} \in \Gamma(T\tilde{B})$, $\beta \in C^{\infty}(\overline{B})$ and $\mathrm{E}_{1}, \mathrm{F}_{1} \in \Gamma(T\overline{B})$. Then the following formulae hold:
\begin{enumerate}
\item[(1)] $\overline{R}(\mathrm{E}_{1}, \mathrm{F}_{1})\xi = \overline{g}(\mathrm{F}_{1}, \xi)\mathrm{E}_{1} - \overline{g}(\mathrm{E}_{1}, \xi)\mathrm{F}_{1}$.
\item[(2)] $\overline{R}(\xi, \mathrm{E}_{1})\mathrm{F}_{1} = \overline{g}(\mathrm{E}_{1}, \mathrm{F}_{1})\xi - \overline{g}(\xi, \mathrm{F}_{1})\mathrm{E}_{1}$.
\item[(3)] $\overline{R}(\phi \mathrm{E}_{1}, \xi)\mathrm{F}_{1} = \overline{g}(\xi, \mathrm{F}_{1})\phi \mathrm{E}_{1} - \overline{g}(\phi \mathrm{E}_{1}, \mathrm{F}_{1})\xi$.
\item[(4)] $\overline{R}(\mathrm{E}_{1}, \phi \mathrm{F}_{1})\xi + \overline{R}(\xi, \mathrm{E}_{1})\phi \mathrm{F}_{1} = - \overline{R}(\phi \mathrm{F}_{1}, \xi)\mathrm{E}_{1}$.
\item[(5)] The sectional curvature $\overline{\mathcal{K}}$ for a plane section containing $\xi$ is equal to $\overline{g}(\mathrm{E}_{1}, \mathrm{E}_{1}) - \overline{g}(\mathrm{E}_{1}, \xi)^{2}$ at every point of $\overline{B}$, that is,
    \begin{eqnarray*}
    \overline{\mathcal{K}}(\mathrm{E}_{1} \wedge \xi) = \overline{g}(\overline{R}(\mathrm{E}_{1}, \xi)\xi, \mathrm{E}_{1}) = \overline{g}(\mathrm{E}_{1}, \mathrm{E}_{1}) - \overline{g}(\mathrm{E}_{1}, \xi)^{2}.
    \end{eqnarray*}
\end{enumerate}
\end{proposition}

\begin{proof}
By using (\ref{2.1}) and straightforward computation, we have our assertions (1), (2) and (3).
We get (4) easily by adding (1) and (2). Furthermore, we evaluate $\overline{g}(\overline{R}(\mathrm{E}_{1}, \mathrm{F}_{1})\mathrm{G}_{1}, \mathrm{H}_{1})$ for $\mathrm{H}_{1} =\mathrm{E}_{1}$ and $\mathrm{F}_{1} = \mathrm{G}_{1} =\xi$ and use (\ref{2.1}), we get our last assertion (5).
\end{proof}

\begin{proposition}\label{pr}
Let $(\tilde{B}, \tilde{\nabla} = \nabla^{\tilde{g}} + \mathbb{K}, \tilde{g}, J)$ and $(\overline{B}, \overline{\nabla} = \nabla^{\overline{g}} + \overline{\mathbb{K}}, \overline{g}, \phi, \xi)$ be a holomorphic statistical manifold, and the Kenmotsu statistical manifold as in  Proposition \ref{1p}, respectively. If $\overline{B}$ is of constant $\phi-$sectional curvature $\overline{c}$ and dim$(\overline{B})=2s+1$, then the Ricci tensor $\overline{Ric}$ of $\overline{B}$ has the following forms:
\begin{enumerate}
\item[(1)] $\overline{Ric}(\mathrm{E}_{1}, \mathrm{F}_{1}) =  t_{1} \hspace{0.2 cm} \overline{g}(\mathrm{E}_{1}, \mathrm{F}_{1}) + t_{2} \hspace{0.2 cm} \overline{g}(\mathrm{E}_{1}, \xi)\overline{g}(\mathrm{F}_{1}, \xi)$, where
    \begin{eqnarray*}
 t_{1} = \frac{\overline{c}(s +1)-3s +1}{2}, \hspace{0.5 cm} \textrm{and} \hspace{0.5 cm} t_{2} = \frac{-(\overline{c}+1)(s +1)}{2}.
 \end{eqnarray*}
Moreover, $\overline{B}$ is not a Ricci-flat statistical manifold.
\item[(2)] $\overline{Ric}(\mathrm{E}_{1}, \xi) = -2s \overline{g}(\mathrm{E}_{1}, \xi)$.
\item[(3)] $\overline{Ric}(\phi \mathrm{E}_{1}, \phi \mathrm{F}_{1}) = \overline{Ric}(\mathrm{E}_{1}, \mathrm{F}_{1}) + 2s \overline{g}(\mathrm{E}_{1}, \xi)\overline{g}(\mathrm{F}_{1}, \xi)$.
\end{enumerate}
\end{proposition}

\begin{proof}
From (\ref{2.1}) and our assumptions, we have
\begin{eqnarray*}
\overline{g}(\overline{R}(\mathrm{E}_{1}, \mathrm{F}_{1})\mathrm{G}_{1}, \mathrm{H}_{1})
= \overline{g}(\overline{R}^{\overline{g}}(\mathrm{E}_{1}, \mathrm{F}_{1})\mathrm{G}_{1}, \mathrm{H}_{1}).
\end{eqnarray*}
It is known that $\overline{R}^{\overline{g}}$ is written as the right hand side of (1)-(3) (see \cite{11}). We see that the Ricci curvature of $\overline{B}$ never vanish. Thus, $\overline{B}$ can not be a Ricci-flat statistical manifold.
\end{proof}

\begin{example}\label{ex2}
We recall Examples 3.3 and 3.10 of \cite{f2}. $(\mathbb{H}^{2s+1}, \overline{\nabla} = \nabla^{\overline{g}} + \overline{\mathbb{K}}, \overline{g}, \phi, \xi)$ is a Kenmotsu statistical manifold of constant $\phi-$sectional curvature $\overline{c}= -1$. By  Proposition \ref{pr}, we conclude that $\mathbb{H}^{2s+1}$ is not a Ricci-flat statistical manifold.
\end{example}

\begin{proposition}\label{pp2}
Let $(\overline{B}(\overline{c}), \overline{\nabla}, \overline{g}, \phi, \xi)$ be a Kenmotsu statistical manifold of constant $\phi-$sectional curvature $\overline{c}$ and $(N, \nabla, g)$ be a statistical submanifold in $\overline{B}(\overline{c})$ such that $\xi$ is tangent to $N$ and $\phi(TN) \subset TN$. Suppose that
\begin{enumerate}
\item[(1)] $\overline{c}= -1$;
\item[(2)] $h(\mathrm{E}_{1}, \mathrm{F}_{1}) = g(\mathrm{E}_{1}, \mathrm{F}_{1})\mathcal{H}$ and $h^{\ast}(\mathrm{E}_{1}, \mathrm{F}_{1}) = g(\mathrm{E}_{1}, \mathrm{F}_{1})\mathcal{H}^{\ast}$, for any $\mathrm{E}_{1}, \mathrm{F}_{1} \in \Gamma(TN)$.
\end{enumerate}
Then $N$ is a statistical manifold of constant curvature $g(\mathcal{H}, \mathcal{H}^{\ast})-1$ whenever $g(\mathcal{H}, \mathcal{H}^{\ast})$ is a constant.
\end{proposition}

\begin{proof}
By using equation (\ref{i}) and given condition (2), we get
\begin{eqnarray*}
R(\mathrm{E}_{1}, \mathrm{F}_{1})\mathrm{G}_{1} &=& \frac{\overline{c}-3}{4} [g(\mathrm{F}_{1}, \mathrm{G}_{1})\mathrm{E}_{1} - g(\mathrm{E}_{1}, \mathrm{G}_{1})\mathrm{F}_{1}] \\
&&+ g(\mathcal{H}, \mathcal{H}^{\ast}) [g(\mathrm{F}_{1}, \mathrm{G}_{1})\mathrm{E}_{1} - g(\mathrm{E}_{1}, \mathrm{G}_{1})\mathrm{F}_{1} ].
\end{eqnarray*}
Taking into account of condition (1), we obtain our assertion, that is,
\begin{eqnarray*}
R(\mathrm{E}_{1}, \mathrm{F}_{1})\mathrm{G}_{1} = (g(\mathcal{H}, \mathcal{H}^{\ast})-1) [g(\mathrm{F}_{1}, \mathrm{G}_{1})\mathrm{E}_{1} - g(\mathrm{E}_{1}, \mathrm{G}_{1})\mathrm{F}_{1} ],
\end{eqnarray*}
for any $\mathrm{E}_{1}, \mathrm{F}_{1}, \mathrm{G}_{1} \in \Gamma(TN)$.
\end{proof}

\section{Chen-Ricci Inequality for Statistical Submanifolds}

\subsection*{Proof of Theorem \ref{th}:}

We choose $\{e_{1}, \dots, e_{k+1}\}$ as the orthonormal frame of $T_{p}N$ such that $e_{1} = \mathrm{E}_{1}$ and $||\mathrm{E}_{1}|| = 1$, and $\{e_{k+2}, \dots, e_{2s+1}\}$ as the the orthonormal frame of $T_{p}N^{\perp}$. Then by formulae (\ref{i111})-(\ref{r2}), we have
\begin{eqnarray*}
2\overline{S}(e_{1}, e_{i}, e_{1}, e_{i}) &=& 2S(e_{1}, e_{i}, e_{1}, e_{i}) - g(h(e_{1}, e_{1}), h^{\ast }(e_{i}, e_{i}))\\
&& - g(h^{\ast }(e_{1}, e_{1}), h(e_{i}, e_{i}))
+ 2g(h(e_{1}, e_{i}), h^{\ast }(e_{1}, e_{i}))\\
&=& 2S(e_{1}, e_{i}, e_{1}, e_{i}) - \{4g(h^{0}(e_{1}, e_{1}), h^{0}(e_{i}, e_{i}))\\
&&- g(h(e_{1}, e_{1}), h(e_{i}, e_{i})) - g(h^{\ast }(e_{1}, e_{1}), h^{\ast }(e_{i}, e_{i}))\\
&&- 4g(h^{0}(e_{1}, e_{i}), h^{0}(e_{1}, e_{i})) + g(h(e_{1}, e_{i}), h(e_{1}, e_{i}))\\
&&+ g(h^{\ast }(e_{1}, e_{i}), h^{\ast }(e_{1}, e_{i}))\}\\
&=& 2S(e_{1}, e_{i}, e_{1}, e_{i}) - 4\sum_{r=k+1}^{s}(h^{0r}_{11}h^{0r}_{ii} - (h^{0r}_{1i})^{2})\\
&&+ \sum_{r=k+1}^{s}(h^{r}_{11}h^{r}_{ii} - (h^{r}_{1i})^{2}) + \sum_{r=k+1}^{s}(h^{\ast r}_{11}h^{\ast r}_{ii} - (h^{\ast r}_{1i})^{2}),
\end{eqnarray*}
where we have used the notation
\begin{eqnarray*}
\overline{S}(\mathrm{E}_{1},\mathrm{F}_{1},\mathrm{G}_{1},\mathrm{H}_{1}) = \overline{g}(\overline{S}(\mathrm{E}_{1},\mathrm{F}_{1})\mathrm{H}_{1},\mathrm{G}_{1}).
\end{eqnarray*}
Summing over $2 \leq i \leq k+1$ and using (\ref{d}), we arrive at
\begin{eqnarray*}
&&2\{\frac{3(\overline{c}+1)}{4}||\mathcal{P} \mathrm{E}_{1}||^{2} + \frac{k}{4}[(\overline{c}+1)(1-g^{2}(\mathrm{E}_{1},\xi))-4]\} \\
&=& 2Ric^{\nabla, \nabla^{\ast }}(\mathrm{E}_{1}) - 4\sum_{r=k+2}^{2s+1}\sum_{i=2}^{k+1}(h^{0r}_{11}h^{0r}_{ii} - (h^{0r}_{1i})^{2}) \\
&&+ \sum_{r=k+2}^{2s+1}\sum_{i=2}^{k+1}(h^{r}_{11}h^{r}_{ii} - (h^{r}_{1i})^{2}) + \sum_{r=k+2}^{2s+1}\sum_{i=2}^{k+1}(h^{\ast r}_{11}h^{\ast r}_{ii} - (h^{\ast r}_{1i})^{2}),
\end{eqnarray*}
where $Ric^{\nabla, \nabla^{\ast }}(\mathrm{E}_{1})$ denotes the Ricci curvature of $N$ with respect to $\nabla$ and $\nabla^{\ast }$ at $p$. Further, we derive
\begin{eqnarray}\label{g2}
&&2Ric^{\nabla, \nabla^{\ast }}(\mathrm{E}_{1}) - 2\{\frac{3(\overline{c}+1)}{4}||\mathcal{P} \mathrm{E}_{1}||^{2} + \frac{k}{4}[(\overline{c}+1)(1-g^{2}(\mathrm{E}_{1},\xi))-4]\} \nonumber \\
&=& 4\sum_{r=k+2}^{2s+1}\sum_{i=2}^{k+1}(h^{0r}_{11}h^{0r}_{ii} - (h^{0r}_{1i})^{2}) - \sum_{r=k+2}^{2s+1}\sum_{i=2}^{k+1}(h^{r}_{11}h^{r}_{ii} - (h^{r}_{1i})^{2})\nonumber \\
&&- \sum_{r=k+2}^{2s+1}\sum_{i=2}^{k+1}(h^{\ast r}_{11}h^{\ast r}_{ii} - (h^{\ast r}_{1i})^{2}).
\end{eqnarray}
By Gauss equation with respect to $\overline{\nabla}^{\overline{g}}$, it follows that
\begin{eqnarray*}
&&Ric^{0}(\mathrm{E}_{1}) - \{\frac{3(\overline{c}+1)}{4}||\mathcal{P} \mathrm{E}_{1}||^{2} + \frac{k}{4}[(\overline{c}+1)(1-g^{2}(\mathrm{E}_{1},\xi))-4]\}\\
&=& \sum_{r=k+2}^{2s+1}\sum_{i=2}^{k+1}(h^{0r}_{11}h^{0r}_{ii} - (h^{0r}_{1i})^{2}).
\end{eqnarray*}
Substituting into (\ref{g2}), we arrive at
\begin{eqnarray*}
&&2Ric^{\nabla, \nabla^{\ast }}(\mathrm{E}_{1}) - 2\{\frac{3(\overline{c}+1)}{4}||\mathcal{P} \mathrm{E}_{1}||^{2} + \frac{k}{4}[(\overline{c}+1)(1-g^{2}(\mathrm{E}_{1},\xi))-4]\} \\
&=& 4Ric^{0}(\mathrm{E}_{1}) - \{3(\overline{c}+1)||\mathcal{P} \mathrm{E}_{1}||^{2} + k[(\overline{c}+1)(1-g^{2}(\mathrm{E}_{1},\xi))-4]\} \nonumber \\
&&- \sum_{r=k+2}^{2s+1}\sum_{i=2}^{k+1}(h^{r}_{11}h^{r}_{ii} - (h^{r}_{1i})^{2}) - \sum_{r=k+2}^{2s+1}\sum_{i=2}^{k+1}(h^{\ast r}_{11}h^{\ast r}_{ii} - (h^{\ast r}_{1i})^{2}).
\end{eqnarray*}
On simplifying the previous relation, we get
\begin{eqnarray}\label{g3}
&&-Ric^{\nabla, \nabla^{\ast }}(\mathrm{E}_{1}) - \{\frac{3(\overline{c}+1)}{4}||\mathcal{P} \mathrm{E}_{1}||^{2} + \frac{k}{4}[(\overline{c}+1)(1-g^{2}(\mathrm{E}_{1},\xi))-4]\} \nonumber \\
&&+ 2Ric^{0}(\mathrm{E}_{1}) \nonumber \\
&=& \sum_{r=k+2}^{2s+1}\sum_{i=2}^{k+1}(h^{r}_{11}h^{r}_{ii} - (h^{r}_{1i})^{2}) + \sum_{r=k+2}^{2s+1}\sum_{i=2}^{k+1}(h^{\ast r}_{11}h^{\ast r}_{ii} - (h^{\ast r}_{1i})^{2})\nonumber \\
&\leq& \sum_{r=k+2}^{2s+1}\sum_{i=2}^{k+1}h^{r}_{11}h^{r}_{ii} + \sum_{r=k+2}^{2s+1}\sum_{i=2}^{k+1}h^{\ast r}_{11}h^{\ast r}_{ii}.
\end{eqnarray}
Let us define the quadratic form $\theta_{r}, \theta_{r}^{\ast} : \mathbb{R}^{k+1} \rightarrow \mathbb{R}$ by
$$\theta_{r}(h_{11}^{r}, h_{22}^{r}, \dots, h_{kk}^{r}) = \sum_{r=k+2}^{2s+1}\sum_{i=2}^{k+1}h^{r}_{11}h^{r}_{ii},$$
$$\theta_{r}^{\ast }(h_{11}^{\ast r}, h_{22}^{\ast r}, \dots, h_{kk}^{\ast r}) = \sum_{r=k+2}^{2s+1}\sum_{i=2}^{k+1}h^{\ast r}_{11}h^{\ast r}_{ii}.$$
We consider the constrained extremum problem $\max \theta_{r}$ subject to $$Q : \sum_{i=1}^{k+1}h^{r}_{ii} = a^{r},$$ where $a^{r}$ is a real constant. The gradient vector field of the function $\theta_{r}$ is given by $$grad \hspace{0.1 cm} \theta_{r} = (\sum_{i=2}^{k+1}h^{r}_{ii}, h^{r}_{11}, h^{r}_{11}, \dots, h^{r}_{11}).$$ For an optimal solution $p = (h_{11}^{r}, h_{22}^{r}, \dots h_{kk}^{r})$ of the problem in question, the vector $grad \hspace{0.1 cm} \theta_{r}$ is normal to $Q$ at the point $p$. It follows that $$h^{r}_{11} = \sum_{i=2}^{k+1}h^{r}_{ii} = \frac{a^{r}}{2}.$$
Now, we fix $x \in Q$. The bilinear form $\pi : T_{x}Q \times T_{x}Q \rightarrow \mathbb{R}$ has the following expression:
\begin{eqnarray*}
\pi(\mathrm{E}_{1}, \mathrm{F}_{1}) = Hess_{\theta_{r}}(\mathrm{E}_{1}, \mathrm{F}_{1}) + <h^{'}(\mathrm{E}_{1}, \mathrm{F}_{1}), (grad \hspace{0.1 cm} \theta_{r})(x)>,
\end{eqnarray*}
where $h^{'}$ denotes the second fundamental form of $Q$ in $\mathbb{R}^{k+1}$ and $<\cdot,\cdot>$ denotes the standard inner product on $\mathbb{R}^{k+1}$. The Hessian matrix of $\theta_{r}$ is given by
\begin{eqnarray*}
Hess_{\theta_{r}} = \left(
 \begin{array}{ccccc}
    0 & 1 & \dots & 1 \\
    1 & 0 & \dots& 0 \\
    \vdots & \vdots \ & \ddots & \vdots \\
    1 & 0  & \dots &  0 \\
    1 & 0  & \dots &  0
  \end{array}
\right).
\end{eqnarray*}
We consider a vector $\mathrm{E}_{1} \in T_{x}Q$, which satisfies a relation $\mathrm{E}_{2} + \dots + \mathrm{E}_{k+1} = -\mathrm{E}_{1}$. As $h^{'} = 0$ in $\mathbb{R}^{k+1}$, we get
\begin{eqnarray*}
\pi(\mathrm{E}_{1}, \mathrm{E}_{1}) = Hess_{\theta_{r}}(\mathrm{E}_{1}, \mathrm{E}_{1}) &=& 2 \mathrm{E}_{1}(\mathrm{E}_{2} + \dots + \mathrm{E}_{k+1}) \\
&=& (\mathrm{E}_{1} + \mathrm{E}_{2} + \dots + \mathrm{E}_{k+1})^{2} \\
&&- (\mathrm{E}_{1})^{2} - (\mathrm{E}_{2} + \dots + \mathrm{E}_{k+1})^{2}\\
&=& - 2(\mathrm{E}_{1})^{2} \\
&\leq& 0.
\end{eqnarray*}
However, the point $p$ is the only optimal solution, i.e., the global maximum point of problem. Thus, we obtain
\begin{eqnarray}\label{g4}
\theta_{r} \leq \frac{1}{4}(\sum_{i=1}^{k+1} h_{ii}^{r})^{2} = \frac{(k+1)^{2}}{4}(\mathcal{H}^{r})^{2}.
\end{eqnarray}
Next, we deal with the constrained extremum problem $\max \theta_{r}^{\ast}$ subject to $$Q^{\ast} : \sum_{i=1}^{k+1}h^{\ast r}_{ii} = a^{\ast r},$$ where $a^{\ast r}$ is a real constant. By similar arguments as above, we find
\begin{eqnarray}\label{g410}
\theta_{r}^{\ast} \leq \frac{1}{4}(\sum_{i=1}^{k+1} h_{ii}^{\ast r})^{2} = \frac{(k+1)^{2}}{4}(\mathcal{H}^{\ast r})^{2}.
\end{eqnarray}
On combining (\ref{g3}), (\ref{g4}) and (\ref{g410}), we get our desired inequality (\ref{r}). Moreover, the vector field $\mathrm{E}_{1}$ satisfies the equality case if and only if
$$h_{1i}^{r} = 0,  \textrm{             } h_{1i}^{\ast r} = 0, \textrm{                 } i \in \{2, \dots, k+1\}$$
and
$$h^{r}_{11} = (\frac{k+1}{2})\mathcal{H},  \textrm{             } h^{\ast r}_{11} = (\frac{k+1}{2})\mathcal{H}^{\ast}$$
Thus, it proves our assertion.

\section{Conclusions and Remarks}

\begin{remark}
We recall some important results given by Furuhata et al. \cite{f2} and add more new results to a Kenmotsu statistical manifold (the warped product of a holomorphic statistical manifold and a line) of constant $\phi-$sectional curvature with some related examples. It is known that the Ricci tensor of a statistical manifold is not symmetric, unlike the Riemannian case where the Ricci tensor of the Riemannian connection is symmetric and has a precise geometric and physical meaning. For a torsion-free affine connection on a simply connected $n-$manifold, the Ricci tensor is symmetric if and only if the connection preserves a volume $n-$form. But this is not true in the case of Kenmotsu statistical manifolds. So, it is unnatural to consider the condition that Ricci tenor to be proportional with the metric tensor. We also conclude that a Kenmotsu statistical manifold of constant $\phi-$sectional curvature can not be a Ricci-flat statistical manifold (by Proportion \ref{pr} (1)). For instance, a Kenmotsu statistical manifold $(\mathbb{H}^{2s+1}, \overline{\nabla} = \nabla^{\overline{g}} + \overline{\mathbb{K}}, \overline{g}, \phi, \xi)$ of constant $\phi-$sectional curvature $-1$ is not a Ricci-flat statistical manifold.
\end{remark}

\begin{remark}
As we know that $2\mathcal{H}^{0} = \mathcal{H} + \mathcal{H}^{\ast }$. Then the inequality (\ref{r}) can be rewritten as

\begin{corollary}
Let $(\overline{B}(\overline{c}), \overline{\nabla}, \overline{g}, \phi, \xi)$ be a $(2s+1)-$dimensional Kenmotsu statistical manifold of constant $\phi-$sectional curvature $\overline{c}$ and $(N, \nabla, g)$ be an $(k+1)-$dimensional statistical submanifold in $\overline{B}(\overline{c})$. Then for each unit vector $\mathrm{E}_{1} \in T_{p}N$, we have
\begin{eqnarray*}
Ric^{\nabla, \nabla^{\ast }}(\mathrm{E}_{1}) &\geq& 2Ric^{0}(\mathrm{E}_{1}) - \{\frac{3(\overline{c}+1)}{4}||\mathcal{P} \mathrm{E}_{1}||^{2} + \frac{k}{4}[(\overline{c}+1)(1-g^{2}(\mathrm{E}_{1},\xi))-4]\} \nonumber \\
&&- \frac{(k+1)^{2}}{2} ||\mathcal{H}^{0}||^{2} + \frac{(k+1)^{2}}{4}g(\mathcal{H}, \mathcal{H}^{\ast}).
\end{eqnarray*}
\end{corollary}

\begin{corollary}
Let $(\overline{B}(\overline{c}), \overline{\nabla}, \overline{g}, \phi, \xi)$ be a $(2s+1)-$dimensional Kenmotsu statistical manifold of constant $\phi-$sectional curvature $\overline{c}$ and $(N, \nabla, g)$ be an $(k+1)-$dimensional statistical submanifold in $\overline{B}(\overline{c})$. Suppose that $N$ is minimal with respect to $\overline{\nabla}^{\overline{g}}$, then for each unit vector $\mathrm{E}_{1} \in T_{p}N$, we have
\begin{eqnarray*}
Ric^{\nabla, \nabla^{\ast }}(\mathrm{E}_{1}) &\geq& 2Ric^{0}(\mathrm{E}_{1}) - \{\frac{3(\overline{c}+1)}{4}||\mathcal{P} \mathrm{E}_{1}||^{2} + \frac{k}{4}[(\overline{c}+1)(1-g^{2}(\mathrm{E}_{1},\xi))-4]\} \nonumber \\
&& + \frac{(k+1)^{2}}{4}g(\mathcal{H}, \mathcal{H}^{\ast}).
\end{eqnarray*}
\end{corollary}

Further, we derive

\begin{corollary}
Let $(\overline{B}(\overline{c}), \overline{\nabla}, \overline{g}, \phi, \xi)$ be a $(2s+1)-$dimensional Kenmotsu statistical manifold of constant $\phi-$sectional curvature $\overline{c}$ and $(N, \nabla, g)$ be an $(k+1)-$dimensional statistical submanifold in $\overline{B}(\overline{c})$.
\begin{enumerate}
\item[(1)] For each unit vector $\mathrm{E}_{1} \in T_{p}N$ orthogonal to $\xi$, we have
\begin{eqnarray}\label{rrr}
Ric^{\nabla, \nabla^{\ast }}(\mathrm{E}_{1}) &\geq& 2Ric^{0}(\mathrm{E}_{1}) - \{\frac{3(\overline{c}+1)}{4}||\mathcal{P} \mathrm{E}_{1}||^{2} + \frac{(\overline{c}-3)k}{4}\} \nonumber \\
&&- \frac{(k+1)^{2}}{8} [||\mathcal{H}||^{2} + ||\mathcal{H}^{\ast }||^{2}].
\end{eqnarray}
\begin{enumerate}
\item[(1.1)] If $N$ is invariant, then
\begin{eqnarray}\label{rr1}
Ric^{\nabla, \nabla^{\ast }}(\mathrm{E}_{1}) &\geq& 2Ric^{0}(\mathrm{E}_{1}) - \frac{\overline{c}(k+3)+3(1-k)}{4} \nonumber \\
&&- \frac{(k+1)^{2}}{8} [||\mathcal{H}||^{2} + ||\mathcal{H}^{\ast }||^{2}].
\end{eqnarray}
\item[(1.2)] If $N$ is anti-invariant, then
\begin{eqnarray}\label{rr2}
Ric^{\nabla, \nabla^{\ast }}(\mathrm{E}_{1}) &\geq& 2Ric^{0}(\mathrm{E}_{1}) - \frac{k(\overline{c}-3)}{4} \nonumber \\
&&- \frac{(k+1)^{2}}{8} [||\mathcal{H}||^{2} + ||\mathcal{H}^{\ast }||^{2}].
\end{eqnarray}
\end{enumerate}
\item[(2)] Moreover, the equality holds in the inequalities (\ref{rrr})-(\ref{rr2}) if and only if $$2h(\mathrm{E}_{1}, \mathrm{E}_{1}) = (k+1) \mathcal{H}(p),$$ $$2h^{\ast}(\mathrm{E}_{1}, \mathrm{E}_{1}) = (k+1) \mathcal{H}^{\ast}(p)$$ and $$h(\mathrm{E}_{1}, \mathrm{F}_{1}) =0,$$ $$h^{\ast }(\mathrm{E}_{1}, \mathrm{F}_{1})=0,$$ for all $\mathrm{F}_{1} \in T_{p}N$ orthogonal to $\mathrm{E}_{1}$.
\end{enumerate}
\end{corollary}
\end{remark}

\begin{remark}\label{r4}
By Theorem \ref{p1}, Proposition 3.9 of \cite{f2}, we say that $\overline{c} = -1$ and $(\tilde{B}, \tilde{\nabla}, \tilde{g}, J)$ is of constant holomorphic sectional curvature $0$. Thus, from Theorem \ref{th}, one can easily obtain the following result:

\begin{corollary}
Let $(\tilde{B}, \tilde{\nabla} = \tilde{\nabla}^{\tilde{g}} + \tilde{\mathbb{K}}, \tilde{g}, J)$ and $(\overline{B}, \overline{\nabla} = \overline{\nabla}^{\overline{g}} + \overline{\mathbb{K}}, \overline{g}, \phi, \xi)$ be a holomorphic statistical manifold of constant holomorphic sectional curvature $0$, and be the Kenmotsu statistical manifold of constant $\phi-$sectional curvature $-1$ as in Proposition 3.9 of \cite{f2}, respectivley. If $(N, \nabla, g)$ is a statistical submanifold in $\overline{B}$ with dim$(\overline{B})=2s+1$ and dim$(N) = k+1$, then
\begin{enumerate}
\item[(1)] For each unit vector $\mathrm{E}_{1} \in T_{p}N$, $p \in N$,
\begin{eqnarray}\label{rr}
Ric^{\nabla, \nabla^{\ast }}(\mathrm{E}_{1}) \geq 2Ric^{0}(\mathrm{E}_{1}) + 4 - \frac{(k+1)^{2}}{8} [||\mathcal{H}||^{2} + ||\mathcal{H}^{\ast }||^{2}].
\end{eqnarray}
\item[(2)] Moreover, the equality holds in the inequality (\ref{rr}) if and only if $$2h(\mathrm{E}_{1}, \mathrm{E}_{1}) = (k+1) \mathcal{H}(p),$$ $$2h^{\ast}(\mathrm{E}_{1}, \mathrm{E}_{1}) = (k+1) \mathcal{H}^{\ast}(p)$$ and $$h(\mathrm{E}_{1}, \mathrm{F}_{1}) =0,$$ $$h^{\ast }(\mathrm{E}_{1}, \mathrm{F}_{1})=0,$$ for all $\mathrm{F}_{1} \in T_{p}N$ orthogonal to $\mathrm{E}_{1}$.
\end{enumerate}
\end{corollary}
\end{remark}

\section*{Acknowledgment}
The second author was supported by the grant Project No. 2018-R1D1A1B-05040381 from National Research Foundation.

\end{document}